\documentclass[11pt, a4paper]{amsart}
\usepackage[utf8]{inputenc}
\usepackage{amsfonts, amssymb, amsmath, textcomp, amsthm, xcolor, mathrsfs}
\usepackage{appendix}
\usepackage{url}

\def\C{\mathbb{C}}\def\nint{\mathop{\diagup\kern-13.0pt\int}}
\def\Z{\mathbb{Z}}

\def\beq{\begin{equation}}
\def\endeq{\end{equation}}
\def\bg{\begin{gathered}}
\def\eg{\end{gathered}}

\newtheorem{thm}{Theorem}[section]
\newtheorem{prop}[thm]{Proposition}

\newtheorem{lem}[thm]{Lemma}

\newtheorem{defi}[thm]{Definition}

\newtheorem{rem}[thm]{Remark}
\newtheorem{question}[thm]{Question}

\newtheorem{conj}[thm]{Conjecture}

\title{Mixed-norm Brascamp-Lieb inequalities}
\author{Ruixiang Zhang}
\date{\today}
\thanks{RZ is supported by NSF CAREER DMS-2143989.}

\begin{document}

\maketitle

\begin{abstract}
Christ \cite{christ2001certain} looked at a certain boundedness problem for trilinear operators. In this paper we set up a framework of mixed-norm Brascamp-Lieb inequalities and identify a new interesting regime not covered by the study of classical Brascamp-Lieb. Positive results in \cite{christ2001certain} can be viewed as the first nontrivial progress in this new regime. We will also prove another mixed-norm Brascamp-Lieb inequality, showcasing how one can use a tensor product trick to slightly sharpen Christ's argument and also obtain the endpoint case. We then discuss the connections between mixed-norm Brascamp-Lieb and the Kakeya Conjectures, as well as unique new difficulties for these mixed-norm Brascamp-Lieb inequalities compared to the classical setting.
\end{abstract}

\tableofcontents

\section{Introduction}

\subsection{Purpose of the paper}

The purpose of this paper is twofold. On the concrete side, we will prove the following quadrilinear inequality:

\begin{thm}\label{mixedBLineqforR}
\begin{equation}\label{mainthmBLineq}
    \|f_1 (x+y)^{\frac{2}{7}}f_2 (-x+y)^{\frac{2}{7}}f_3 (y)^{\frac{2}{7}}f_4 (x+3y)^{\frac{1}{7}}\|_{L_x^{\frac{7}{4}} L_y^{\frac{7}{3}}} \lesssim \|f_1\|_{L^1}^{\frac{2}{7}}\|f_2\|_{L^1}^{\frac{2}{7}}\|f_3\|_{L^1}^{\frac{2}{7}}\|f_4\|_{L^1}^{\frac{1}{7}}
\end{equation}
where $f_1, f_2, f_3, f_4: \mathbb{R} \to \mathbb{R}_{\geq 0}$.
\end{thm}

To motivate \eqref{mainthmBLineq}, note that related trilinear inequalities (and some aspects of the multilinear ones through a different but equivalent formulation) were already studied by the work of Christ \cite{christ2001certain}. Our second purpose is to give a slightly different viewpoint and think of Christ's result and \eqref{mainthmBLineq} as ``mixed-norm Brascamp-Lieb inequalities''. We will compare them with classical Brascamp-Lieb inequalities, highlighting their key differences, and explain some connections between mixed-norm Brascamp-Lieb (in $\Bbb{R}^2$) and Kakeya (in $\Bbb{R}^d$).

\subsection{Background of classical Brascamp-Lieb}

we start by explaining the classical Brascamp-Lieb inequalities. 

Suppose $m \in \Bbb{Z}^+$ and $p_1, \ldots, p_m >0$. Suppose $B_1, \ldots, B_m$ are orthogonal projections from $\Bbb{R}^n$ to subspaces $H_1, \ldots, H_m$. If for arbitrary non-negative functions $f_1 , \ldots, f_m$ on $H_1, \ldots, H_m$ respectively we always have 
\begin{equation}\label{BL}
    \int_{\Bbb{R}^n} \prod_{j=1}^m f_j (B_j x)^{p_j} \lesssim \prod_{j=1}^m \left(\int_{H_j} f_j\right)^{p_j},
\end{equation}
then we call \eqref{BL} a \emph{Brascamp-Lieb inequality}. The survey \cite{Zhang2022Brascamp} by the author and the references therein record historical influences of Brascamp-Lieb inequalities and in particular the impact of this inequality in Fourier analysis as a key feature we will return to.

There is a simple condition dictating whether \eqref{BL} can hold:

\begin{thm}[Bennett-Carbery-Christ-Tao \cite{bennett2008brascamp}, see also \cite{bennett2005finite}]\label{BCCTfinite}
Given $(B_1, \ldots, B_m, p_1, \ldots, p_m)$. (\ref{BL}) holds if and only if we have both the \emph{scaling condition}:
\begin{equation}\label{scalingcondBL}
    n = \sum_{j=1}^m p_j \dim H_j
\end{equation}
and the \emph{dimension condition}: 
\begin{equation}\label{dimcondBL}
    \dim V \leq \sum_{j=1}^m p_j \dim (B_j V), \forall \text{ subspace } V \subset \Bbb{R}^n.
\end{equation}
\end{thm}

On the contrary, if we replace the norm on the left hand side of \eqref{BL} by a mixed-norm, in general there is no clean characterization like Theorem \ref{BCCTfinite} on whether such an inequality holds, even in the ``simplest case'' of dimension $n=2$. In what follows, we discuss a rich and unique set of phenomena and connections for mixed-norm inequalities.

\subsection{The mixed-norm case for $n=2$}

In this paper, we only look at the case when dimension $n=2$ and all $\dim H_j$ = 1. In this setting, the underlying mixed-norm question is

\begin{question}\label{BLques1}
Let $m \in \Bbb{Z}^+$. For which $p_1, \ldots, p_m, q, r >0$ and orthogonal projections $B_1, \ldots, B_m$ from $\Bbb{R}^2$ to one-dimensional subspaces $H_1, \ldots, H_m$ is there a finite constant $C$ such that
\begin{equation}\label{BLmix}
    \|\prod_{j=1}^m f_j (B_j (x, y))^{p_j}\|_{L_x^q L_y^r} \leq C \prod_{j=1}^m \|f_j\|_{L^1}^{p_j}
\end{equation}
holds, whenever $f_1 , \ldots, f_m \geq 0$?
\end{question}

We will discuss basic benchmarks of this problem in Section \ref{basicssec}. We can easily reduce to the case where $\sum_{j=1}^m p_j = 1$ and the correct constraint here is $\frac{1}{q} + \frac{1}{r} = 1$. We can also assume no $H_j$ equals the $x$-axis.

In Section \ref{basicssec}, we will see that the case $q<2$ is the most interesting and cannot be obtained from classical Brascamp-Lieb. 
In the $m=3$ case, \eqref{BLmix} for $q<2$ was systematically investigated by Christ \cite{christ2001certain} and he obtained the following theorem: 

\begin{thm}[Follows from Theorem 1 of \cite{christ2001certain}]\label{Christthm}
    Let $m=3$ and $\frac{1}{q}+\frac{1}{r} = 1$.

    (a) If $H_1, H_2, H_3$ are disjoint rational subspaces not equal to $x$-axis, then \eqref{BLmix} always holds for $p_1=p_2=p_3 =\frac{1}{3}$ and some $q<2$.\footnote{A consequence of this for the $k$-linear case $(k>3)$ also follows directly, see \S 7, Remark 4 in \cite{christ2001certain}.}

    (b) If there are two different rational subspaces and one irrational subspace among $H_1, H_2, H_3$, then \eqref{BLmix} cannot hold for $p_1=p_2=p_3 =\frac{1}{3}$ and any $q<2$.
\end{thm}

Note that  \eqref{mainthmBLineq} with  $L_x^q L_y^r$-norm on LHS for some $q<2$ is an immediate consequence of Theorem \ref{Christthm} plus interpolation. Indeed, in \cite{christ2001certain}, Theorem \ref{Christthm} (a) was proved for every $m \geq 3$.

Theorem \ref{Christthm} already looks very different from Theorem \ref{BCCTfinite}, indicating unusual aspects of the $q<2$ case. All $p_j$ are taken to be equal in \cite{christ2001certain}, both for the sake of simplicity and because this setup is related to a question of Kenig and Stein \cite{kenig1999multilinear} (see also  the related \cite{grafakos2001some}). Nevertheless, the more general case where $p_1, p_2, p_3$ are not necessarily the same were also covered in \cite{christ2001certain}, and is actually the way Theorem \ref{Christthm} is proved there. 

\begin{rem}
    For readers familiar with \cite{christ2001certain}, \eqref{BLmix} will appear different in normalization of exponents. Indeed, \cite{christ2001certain} considers inequalities of the form \begin{equation}
        \label{christconv}\|\int \prod_j (g_j \circ B_j)(x, y)\mathrm{d}y\|_{L_x^t} \lesssim \prod_j \|g_j\|_{u_j}.
    \end{equation} Nevertheless, it can be elementarily seen that \eqref{BLmix} and \eqref{christconv} are equivalent, in the sense that given an inequality of one form it is equivalent to the other inequality for some index and norm choices. The $q<2$ case in \eqref{BLmix} corresponds precisely to the $t<1$ case in \eqref{christconv}.
\end{rem}

\subsection{Our results}

Let us summarize what we will prove in  this paper. In \cite{christ2001certain}, to obtain new cases of \eqref{BLmix}, one first obtain a restricted weak type estimate and then relies on interpolation. A priori, this method only works for an open range of $(p_1, p_2, p_3)$ and does not cover the endpoint of that range. The first thing we do (Theorem \ref{transthm}) is to use a tensor product trick to prove that if a restricted weak type estimate is obtained in discrete  analogues of \eqref{BLmix} with a uniform implied constant, it can be automatically upgraded to an honest mixed-norm estimate. As an application example, we prove Theorem \ref{mixedBLineqforR}. The corresponding restricted weak type estimate, Lemma \ref{restrictedweak}, is a generalization of (6) in \cite{katz1999bounds}. Note that this kind of generalization was already used in \cite{christ2001certain} to obtain restricted weak type estimates. Indeed some $L_x^q L_y^r (q<2)$ version of Theorem \ref{mixedBLineqforR} is a corollary of  Theorem \ref{Christthm}.

\cite{christ2001certain} and  the present paper are both related to efforts towards the \emph{Kakeya Conjectures} \cite{bourgain1999dimension, katz1999bounds, katz2002new}. Along this line we explain a natural connection (Theorem \ref{BLmiximplieskakeya}) between \eqref{BLmix} and Kakeya: If one can take $q$ arbitrarily close to $1$, then the Hausdorff version of Kakeya (in \emph{any} dimension) is true. This part essentially follows the framework of Bourgain \cite{bourgain1999dimension} (see also \cite{katz1999bounds, katz2002new}) but is only possible because of very recent progress on quantitative Szemer\'{e}di's Theorem \cite{green2009new, green2017new, leng2023improved, leng2024improved}. The literature only seems to have the Minkowski dimension version\footnote{In the final editing stage of this note, I learned from \cite{tao2025sum} that the same observation was made by Thomas Bloom.}, so we work out necessary modifications for the Hausdorff version.

We will also turn to a technical comparison between \eqref{BLmix} and \eqref{BL}. When $q<2$ there are key differences between natures of the two. As Christ \cite{christ2001certain} pointed out, unlike \eqref{BL}, the extremizers for the $q<2$ case of \eqref{BLmix} are in general not symmetric non-decreasing. We will show another difference based on the techniques in \cite{christ2001certain}: A natural and desirable perturbed version of \eqref{BLmix} can never hold. Thus, new challenges show up if one wants to apply recent developments like \cite{bennett2006multilinear, guth2010endpoint, guth2015short, bennett2018stability} in the mixed-norm setting for $q<2$.

\subsection{Outline of the paper} In Section \ref{basicssec}, we formulate the perturbed Brascamp-Lieb inequality and present some basic knowledge about them. In Section \ref{upgradeWsec}, we prove that knowing the restricted weak type estimates in the discrete setting automatically implies the endpoint mixed-norm Brascamp-Lieb. In Section \ref{restrictedweak1sec}, we prove \eqref{mainthmBLineq} as an application of the above principle. We discuss connections to the Kakeya Conjectures and absence of the perturbed version in Section \ref{relateto Kakeyasection} and leave a technical but standard argument to Appendix \ref{HKakeyasec}.

\section*{Acknowledgement} This work is supported by NSF CAREER DMS--2143989 and Sloan Research Fellowship. The author thanks Ben Green for bringing \cite{Taoentropy} to his attention (see Remark \ref{Greenrem}), Larry Guth for making Remark \ref{katzlem} and Terence Tao for bringing \cite{christ2001certain} to his attention. He would  like to thank Michael Christ for reading through the paper and providing detailed feedback. He would also like to thank Yifan Jing and Dmitrii Zakharov for  helpful comments.

\section{Mixed norm Brascamp-Lieb in dimension $2$: Benchmarks}\label{basicssec}

When $n=2$ and all $\dim H_j = 1$, from the classical Brascamp-Lieb inequality \eqref{BL} alone, there are many situations where one can not immediately answer Question \ref{BLques1}.

Before further explanation, we make some simple remarks.

First, note that for each $j$, on both sides of \eqref{BLmix} we request the powers of $f_j$ to be the same. This is necessary and can be seen by replacing $f_j$ by $\lambda f_j$.

Second, the case $q=r$ for \eqref{BLmix} is  covered by the theory of \eqref{BL}, and the interesting new scenario is when $q \neq r$.

Third, if one $B_j$ is the projection to $x$-axis, then the study of \eqref{BLmix} can be fully reduced to the theory without this $B_j$ by H\"{o}lder. Hence, we will always assume none of $H_j$ is equal to the $x$-axis.

\subsection{Necessary conditions}
In the classical setting \eqref{BL}, the scaling condition forces $\sum_{j=1}^m p_j =2$. By scaling all the $f_j$'s together, we see for \eqref{BLmix} there is again a necessary condition
\begin{equation}\label{originalscalingmix}
    \frac{1}{q}+ \frac{1}{r} = \sum_{j=1}^m p_j.
\end{equation}

Also, by raising \eqref{BL} to the $\beta$-th power ($\beta >0$), we get an equivalent inequality with $p_j, q, r$ replaced by $\beta p_j, \frac{q}{\beta}, \frac{r}{\beta}$, respectively. Because of this homogeneity, from now on let us always assume 
\begin{equation}\label{assumptionofqandr}
    \frac{1}{q}+ \frac{1}{r} = 1.
\end{equation}
without loss of generality.\footnote{Our normalization has the advantage that $q, r \geq 1$ always hold. As an inevitable cost, the readers may find an unharmful inconsistency: With this normalization, \eqref{BL} takes a different but equivalent form.} Now \eqref{originalscalingmix} becomes
\begin{equation}\label{scalingmix}
    \sum_{j=1}^m p_j = 1.
\end{equation}

\subsection{Endpoint cases and new phenomena for $q<2$}
Looking at \eqref{BLmix}, it is natural to ask first if it holds at the endpoints $(q, r) = (\infty, 1)$ and $(1, \infty)$.

At $(q, r) = (\infty, 1)$,  \eqref{BLmix} reads
\begin{equation}
    \int_y \prod_{j=1}^m f_j (x+c_j)^{p_j}\mathrm{d} y \lesssim \prod_{j=1}^m \|f_j\|_{L^1}^{p_j}
\end{equation}
for some $c_j$. This follows from H\"{o}lder whenever we have \eqref{scalingmix}. By mixed-norm H\"{o}lder (Theorem 5.1.2 of \cite{bergh2012interpolation}), we have established \eqref{BLmix} (given \eqref{assumptionofqandr} and \eqref{scalingmix}) when $q \geq 2$.

When $q<2$ and $r>2$, \eqref{BLmix} becomes more interesting. We record an immediate observation.

\begin{lem}[A new necessary condition]\label{newnecessarylem}
    If \eqref{BLmix} holds for some $q<2$ and $r>2$ with \eqref{assumptionofqandr} and \eqref{scalingmix}, then $m>2$.
\end{lem}

\begin{proof}
    Obviously $m>1$.
    
    If $m=2$, we need $H_1\neq H_2$. We test \eqref{BLmix} with each of $f_1$ and $f_2$ being a sum of $N$ characteristic functions of disjoint unit intervals. Then $|f_1 (B_1 x)|^{p_1} |f_2 (B_2 x)|^{p_2}$ becomes the sum of $N^2$ characteristic functions of parallelograms. We can arrange the parallelograms to have disjoint projections to $x$-axis. Then in order for \eqref{BLmix}  to hold we need
    $$(N^2)^{\frac{1}{q}} \lesssim_{B_1, B_2} N^{p_1} N^{p_2} = N.$$
    This is possible only when $q \geq 2$.
\end{proof}

Hence, if we want \eqref{BLmix} for some $q<2$, then we must go beyond the ``bilinear'' setting $m=2$ and require $m$ to be larger. This is in sharp contrast against the non-mixed-norm situation \eqref{BL}. Indeed, when $n=2$ and all $\dim H_j = 1$, \eqref{BL} is not truly new for $m>2$: It can be deduced from \eqref{BL} for $m>2$ and H\"{o}lder. In the general $n$ dimensions when all $\dim H_j = n-1$ (the Loomis-Whitney setting), there is an analogous feature of \eqref{BL} that it is not needed to go beyond the ``multilinear'' setting $m=n$. Indeed, \eqref{BL} for $m>n$ can always be deduced by the $m=n$ case plus H\"{o}lder. However, in the mixed-norm case, we must go to the case $m>n$ to discover new Brascamp-Lieb inequalities \eqref{BLmix}, even in dimension two!

In \cite{christ2001certain}, the case $m=3$ and $p_1=p_2=p_3 = \frac{1}{3}$ were studied carefully and fruitful results such as Theorem \ref{Christthm} were obtained. As noted before, they already have a very different flavor than that of \eqref{BL}.


\section{Relationship with discrete restricted weak type estimates}\label{upgradeWsec}

We change to a slightly broader setting. In this paper, $G$ will always denote a discrete abelian group and $|\cdot|$ on $G$  always denotes the counting function. We will see that a useful way to prove \eqref{BLmix} when the projections are rational is by considering its counterpart where $\mathbb{R}$ is replaced by $G$ first. This perspective has the advantage that as we will prove, it suffices to obtain a restricted weak type estimate with a uniform implied constant.  These are usually  easier to obtain.

To be precise, we start with \eqref{BLmix} with \eqref{assumptionofqandr}, \eqref{scalingmix}. Suppose in addition that all $B_j$ are \emph{rational} in the sense that there are integers $(P_j, Q_j) = 1$ such that 
\begin{equation}\label{expressionBj}
    B_j (x, y) = \frac{1}{\sqrt{P_j^2 + Q_j^2}} (P_j x + Q_j y).
\end{equation} 
We can then define the corresponding restricted weak type estimate on $G^2$.

\begin{defi}\label{restrictedweakgeneral}
For the setup above, we say the \emph{restricted weak type counterpart of \eqref{BLmix}} is satisfied for every torsion-free abelian group $G$ if the following is true:

Let $S_1, \ldots, S_m \subset G$ be finite subsets. If a finite subset $Z \subset G^2$ satisfies (i) $P_j z_1+ Q_j z_2 \in S_j, \forall (z_1, z_2)\in Z, \forall 1 \leq j \leq m$ and (ii) for every element $g\in G$, either there are $\sim M$ elements of the form $(g, \cdot)$ in $Z$, or there are none. Then
\begin{equation}\label{restrictedweakineq}
    |Z| \lesssim M^{1-\frac{q}{r}}\prod_{j=1}^m |S_j|^{p_j q}
\end{equation}
with the implied constant independent of $G$.
\end{defi}

\begin{lem}\label{Discretemainthm}
Suppose that for some choices of $m, p_j, q, r$,  every torsion-free abelian group $G$ satisfies the restricted weak counterpart of \eqref{BLmix} with a uniform implied constant and all $B_j$ being rational projections in the form \eqref{expressionBj}. Then for every such $G$ and finitely supported functions $f_1, \ldots, f_m: G \to \C$ we have the following discrete analogue of \eqref{BLmix}:
\begin{equation}\label{discreteBLineq}
    (\sum_{x \in G} (\sum_{y \in G} \prod_{j=1}^m f_j (P_j x + Q_j y)^{p_j r})^{\frac{q}{r}})^{\frac{1}{q}}\leq \prod_{j=1}^m \|f_j\|_{l^1}^{p_j}.
\end{equation}
\end{lem}

\begin{proof}
    The proof has two steps. First we do a dyadic pigeonholing for each $f_j$ and the following function (in $x$): $H(x) = \sum_{y \in G} \prod_{j=1}^m |f_j (P_j x + Q_j y)|^{p_j r}$ and use the restricted type estimate in the assumption to obtain a weaker version of \eqref{discreteBLineq} with a logarithm factor on the right hand side. Then we use a ``tensor product'' trick that was also used in \cite{katz1999bounds} to obtain the honest \eqref{discreteBLineq}.

    Throughout the proof, every implied constant is allowed to depend on constants $m, p_j, q$ and $r$ and we will not specify these dependencies.
    
    Without loss of generality, we always assume $f_j \geq 1$ at every point in its support. Denote the ``size parameter'' $\|f\|_{total}$ to be $\sum_{j=1}^m \|f_j\|_1$.
    
    Since each $f_j$ is finitely supported, there are only finitely many $x$ such that $H(x) \neq 0$.
    For each dyadic number $c_j \in [1, \|f\|_{total})$,  let
    $$f_{j, c_j} = f_j \cdot 1_{\{g \in G: f_j (g) \in [c_j, 2c_j) \}}$$
    and let $$S_{j, c_j} = \text{supp} f_{j, c_j}.$$ For $c_1, \ldots , c_m \in [1, \|f\|_{total})$, denote $$H_{c_1, \ldots , c_m} (x) = \sum_{y \in G} \prod_{j=1}^m |f_{j, c_j} (P_j x + Q_j y)|^{p_j r}.$$
    Now for another dyadic parameter $b \in [1, \|f\|_{total}^{r+m})$, denote
    $$X_{c_1, \ldots , c_m, b} = \{x\in G: b \leq H_{c_1, \ldots , c_m} (x)\leq 2b\}.$$
    Note that as long as $H_{c_1, \ldots , c_m} (x) \neq 0$, $x$ must belong to some $X_{c_1, \ldots , c_m, b}$ since $$H_{c_1, \ldots , c_m} (x) \leq H(x) \leq \prod_{j=1}^m\|f_j\|_1^{\max\{p_j r, 1\}}.$$
    Now for $c_1, \ldots , c_m$ all fixed,

\begin{eqnarray}\label{ineqwithcs}
    & (\sum_{x \in G} (\sum_{y \in G} \prod_{j=1}^m |f_{j, c_j} (P_j x + Q_j y)|^{p_j r})^{\frac{q}{r}})^{\frac{1}{q}}\nonumber\\
    =  & (\sum_b \sum_{x \in X_{c_1, \ldots, c_m, b}} (\sum_{y \in G} \prod_{j=1}^m |f_{j, c_j} (P_j x + Q_j y)|^{p_j r})^{\frac{q}{r}})^{\frac{1}{q}}\nonumber\\
    \lesssim & \log(\|f\|_{total}+1) \cdot \max_b (\sum_{x \in X_{c_1, \ldots , c_m, b}} (\sum_{y \in G} \prod_{j=1}^m |f_{j, c_j} (P_j x + Q_j y)|^{p_j r})^{\frac{q}{r}})^{\frac{1}{q}}\nonumber\\
    \lesssim & \log(\|f\|_{total}+1) \cdot \max_b b^{\frac{1}{r}} |X_{c_1, \ldots , c_m, b}|^{\frac{1}{q}}.
\end{eqnarray}

By the definition of $S_{j, c_j}$ and $X_{c_1, \ldots , c_m, b}$, we see for each element $x \in X_{c_1, \ldots , c_m, b}$, there are $\sim \frac{b}{\prod_{j=1}^m c_j^{p_j r}}$ different $y$ such that $P_j x + Q_j y \in S_{j, c_j}, \forall 1 \leq j \leq m$. By the restricted weak type counterpart assumption, $$\frac{b}{\prod_{j=1}^m c_j^{p_j r}}\cdot |X_{c_1, \ldots , c_m, b}| \leq \left(\frac{b}{\prod_{j=1}^m c_j^{p_j r}}\right)^{1-\frac{q}{r}} \prod_{j=1}^m |S_{j, c_j}|^{p_j q}$$ and thus $$|X_{c_1, \ldots , c_m, b}| \leq \left(\frac{b}{\prod_{j=1}^m c_j^{p_j r}}\right)^{-\frac{q}{r}} \prod_{j=1}^m |S_{j, c_j}|^{p_j q}.$$

Plugging this in \eqref{ineqwithcs} and note that $c_1 |S_{1, c_1}| \lesssim \|f_1\|_1$, etc., we obtain
\begin{eqnarray}\label{discreteBLineqwithcfixed}
    & (\sum_{x \in G} (\sum_{y \in G} \prod_{j=1}^m |f_{j, c_j} (P_j x + Q_j y)|^{p_j r})^{\frac{q}{r}})^{\frac{1}{q}}\nonumber\\  \lesssim & \log(\|f\|_{total}+1) \cdot \prod_{j=1}^m \|f_j\|_{l^1}^{p_j}.
\end{eqnarray}

Finally we do a sum over $\mathbf{c} = (c_1, \ldots, c_m)$. For short we write $L = \log(\|f\|_{total}+1)$.
\begin{eqnarray}\label{weakdiscreteBLineqprel}
    & (\sum_{x \in G} (\sum_{y \in G} \prod_{j=1}^m |f_{j} (P_j x + Q_j y)|^{p_j r})^{\frac{q}{r}})^{\frac{1}{q}}\nonumber\\\lesssim & (\sum_{x \in G} (L^m \max_{\mathbf{c}}\sum_{y \in G} \prod_{j=1}^m |f_{j, c_j} (P_j x + Q_j y)|^{p_j r})^{\frac{q}{r}})^{\frac{1}{q}}\nonumber\\
    = & L^{\frac{m}{r}}(\sum_{x \in G} (\max_{\mathbf{c}}\sum_{y \in G} \prod_{j=1}^m |f_{j, c_j} (P_j x + Q_j y)|^{p_j r})^{\frac{q}{r}})^{\frac{1}{q}}\nonumber\\
    \lesssim & L^{\frac{m}{r}}(L^m \max_{\mathbf{c}}\sum_{x \in G} (\sum_{y \in G} \prod_{j=1}^m |f_{j, c_j} (P_j x + Q_j y)|^{p_j r})^{\frac{q}{r}})^{\frac{1}{q}}\nonumber\\
    \lesssim & L^{m+1} \prod_{j=1}^m \|f_j\|_{l^1}^{p_j}.
\end{eqnarray}
where we used \eqref{discreteBLineqwithcfixed} in the last line. For convenience we rewrite \eqref{weakdiscreteBLineqprel} as
\begin{eqnarray}\label{weakdiscreteBLineq}
    & (\sum_{x \in G} (\sum_{y \in G} \prod_{j=1}^m |f_{j} (P_j x + Q_j y)|^{p_j r})^{\frac{q}{r}})^{\frac{1}{q}}\nonumber\\  \lesssim & \log(\|f\|_{total}+1)^{m+1}  \prod_{j=1}^m \|f_j\|_{l^1}^{p_j}.
\end{eqnarray}

To get rid of the logarithm factor in \eqref{weakdiscreteBLineq}, we use a tensor product trick. For a natural number $N$ that will be large, we define $F_{j, N} (1 \leq j \leq m)$ on the $N$-fold Cartesian product $G^N$ as: $$F_{j, N} (g_1, \ldots, g_N) = F_j (g_1) F_j (g_2) \cdots F_j (g_N).$$ By chasing definition,

\begin{eqnarray}
    & (\sum_{\mathbf{x} \in G^N} (\sum_{\mathbf{y} \in G^N} \prod_{j=1}^m |F_{j, N} (P_j \mathbf{x} + Q_j \mathbf{y})|^{p_j r})^{\frac{q}{r}})^{\frac{1}{q}}  \nonumber\\  = & \left((\sum_{x \in G} (\sum_{y \in G} \prod_{j=1}^m |f_{j} (P_j x + Q_j y)|^{p_j r})^{\frac{q}{r}})^{\frac{1}{q}}\right)^N.
\end{eqnarray}

Invoke \eqref{weakdiscreteBLineq} for $F_{j, N}$ in place of $f_j$, and note that  $\|F_N\|_{total}\leq \|f\|_{total}^N$, we see

\begin{eqnarray}
    & \left((\sum_{x \in G} (\sum_{y \in G} \prod_{j=1}^m |f_{j} (P_j x + Q_j y)|^{p_j r})^{\frac{q}{r}})^{\frac{1}{q}}\right)^N \nonumber\\  \lesssim & \left(N\log(\|f\|_{total}+1)\right)^{m+1} \cdot \left(\prod_{j=1}^m \|f_j\|_{l^1}^{p_j}\right)^N.
\end{eqnarray}

Let $N \to \infty$ (and keep all $f_{j}$), we deduce \eqref{discreteBLineq}. \end{proof}

\begin{thm}\label{transthm}[A transference principle]
    Suppose that for some choices of $m, p_j, q, r$, every torsion-free abelian group $G$ satisfies the restricted weak counterpart of \eqref{BLmix} with a uniform implied constant and all $B_j$ being rational projections in the form \eqref{expressionBj}. Then \eqref{BLmix} holds.
\end{thm}

\begin{proof}
    Take $G=\mathbb{Z}$ in \eqref{discreteBLineq}, and we see that \eqref{BLmix} holds where each of $f_j, 1 \leq j \leq m$ takes form of a finite sum  $\sum_k a_k 1_{I_k}$ where $I_k$ are disjoint intervals of length $1$ (when we apply \eqref{discreteBLineq}, we can plug in functions $\sum_k 1_{CI_k}$ for large $C$). By rescaling all $f_j$ simultaneously, \eqref{mainthmBLineq} also holds when each $f_j$ is a finite linear combination of characteristic functions of disjoint intervals of equal length. Now \eqref{mainthmBLineq} follows from approximation and Fatou's lemma.
\end{proof}

\section{The restricted weak type estimate for \eqref{mainthmBLineq}}\label{restrictedweak1sec}

As an example of applying Theorem \ref{transthm}, we use it to prove Theorem \ref{mixedBLineqforR} based on a restricted weak type estimate (Lemma \ref{restrictedweak} below). Lemma \ref{restrictedweak} already exists in the literature (see Tao's unpublished note \cite{Taoentropy} and Remark \ref{Greenrem} below), and we include a proof here for self-containment. It is obtained in the same way as \cite[(3)]{christ2001certain} was obtained, and is a stronger version of the following Proposition \ref{restrictedweaksp}  obtained in Katz-Tao's work \cite{katz1999bounds} and their proofs share much in common.
One key point of the present paper is that we further show that the lemma can be upgraded to its endpoint version (Theorem \ref{mixedBLineqforR}) by Theorem \ref{transthm}.

We first recall
\begin{prop}[Katz-Tao, (6) in \cite{katz1999bounds}]\label{restrictedweaksp}
Let $G$ be an abelian group without order $2$ elements. For finite  $S_1, S_2, S_3, S_4 \subset G$ and a finite  $Z \subset G^2$ with every two elements in $S$ having different first coordinates, we have
\begin{equation}
    |\{(x, y) \in Z: x+y \in S_1, -x+y \in S_2, y \in S_3 \text{ and } x+3y \in S_4\}| \leq |S_1|^{\frac{1}{2}}|S_2|^{\frac{1}{2}}|S_3|^{\frac{1}{2}}|S_4|^{\frac{1}{4}}.
\end{equation}
\end{prop}

We remark that the form of Katz-Tao's theorem is slightly different. But their conclusion can be  seen equivalent via a linear transform.

We show next that Katz-Tao's proof of Proposition \ref{restrictedweaksp} leads to a stronger restricted weak type mixed-norm estimate. We set up some notations first: For $(a_1, a_2) \in G^2$, define the projection $\pi_{a_1, a_2}: G^2 \to G$ as $$\pi_{a_1, a_2} (x, y) = a_1 x + a_2 y.$$

\begin{lem}\label{restrictedweak}
Let $G$ be an abelian group without order $2$ elements and $S_1, S_2, S_3, S_4 \subset G$ be finite subsets. If a finite subset $Z \subset G^2$ satisfies (i) $\pi_{1, 1} (z) \in S_1$, $\pi_{-1, 1} (z) \in S_2$, $\pi_{0,1} (z) \in S_3$ and $\pi_{1, 3} (z) \in S_4$, $\forall z\in Z$ and (ii) every element in $\pi_{1, 0} (Z)$ has  $\leq m$ preimages in $Z$ with respect to the map $\pi_{1, 0} (Z)$. Then
\begin{equation}\label{restrictedweakineq}
    |Z| \leq m^{\frac{1}{4}}|S_1|^{\frac{1}{2}}|S_2|^{\frac{1}{2}}|S_3|^{\frac{1}{2}}|S_4|^{\frac{1}{4}}.
\end{equation}
\end{lem}

\begin{proof}
This proof is not too different from the proof of \cite[(3)]{christ2001certain} and has a lot of similarities with The proof of \cite[(6)]{katz1999bounds}. For many steps, we will only sketch those and refer the readers to the counterpart in these papers.

First, we see there are $\geq \frac{|Z|^2}{|S_1|}$ pairs $(z_1, z_2)$ such that $\pi_{1, 1} (z_1) = \pi_{1, 1} (z_2)$ like how (13) is deduced in \cite{katz1999bounds} (from their Lemma 2.1). Applying Lemma 2.1 in \cite{katz1999bounds} again just like the deduction of (18) in \cite{katz1999bounds}, we see there are $\geq \frac{|Z|^4}{|S_1|^2 |S_2| |S_3|}$ tuples $(z_1, z_2, z_3, z_4)$ such that $$\pi_{1, 1} (z_1) = \pi_{1, 1} (z_2), \pi_{1, 1} (z_3) = \pi_{1, 1} (z_4), \pi_{-1, 1} (z_1) = \pi_{-1, 1} (z_3) \text{ and } \pi_{0, 1} (z_2) = \pi_{0, 1} (z_4).$$

For each $4-$tuple $(z_1, z_2, z_3, z_4)$ obtained above, we name their coordinates as $z_1 = (x_1, y_1)$, etc. Now we claim two things:

\textbf{Claim I:} For each such tuple, their $\pi_{1, 0} (z_1)$ is uniquely determined by $\pi_{-1, 1} (z_2), \pi_{0, 1} (z_3)$ and $\pi_{1, 3} (z_4)$.

Indeed, knowing all the latter three, we can deduce 
\begin{eqnarray}
2x_1 = & (x_1 + y_1) - (-x_1 + y_1)\nonumber\\
= & (x_2 + y_2) - (-x_3 + y_3)\nonumber\\
= & -(-x_2 + y_2) - 2y_3 + (2y_2) + (x_3 + y_3)\nonumber\\
= & -(-x_2 + y_2) - 2y_3 + (2y_4) + (x_4 + y_4)\nonumber\\
= & -(-x_2 + y_2) - 2y_3 + (x_4 + 3y_4)
\end{eqnarray}
and divide it by $2$ we know $x_1 =\pi_{1, 0} (z_1)$.

\textbf{Claim II:} Knowing $\pi_{-1, 1} (z_2), \pi_{0, 1} (z_3)$ and $\pi_{1, 3} (z_4)$ and $z_1$ will uniquely determine such a tuple.

Indeed, since we know $x_1 + y_1 = x_2 + y_2$ and also $x_2 - y_2$, we know $z_2$. Since we know $x_1 - y_1 = x_3 - y_3$ and also $y_3$, we know $z_3$. Finally, these two allow us to know $y_4 = y_2$ and $x_4 + y_4 = x_3 + y_3$ and hence $z_4$.

By the two claims, we know the number of all such $4$-tuples is $\leq m |S_2| |S_3| |S_4|$ (since knowing $\pi_{-1, 1} (z_2), \pi_{0, 1} (z_3)$ and $\pi_{1, 3} (z_4)$ enables us know $\pi_{1, 0} (z_1)$, thus limiting the choice of $z_1$ to $\leq m$ possibilities by assumption). Combining this upper bound with the lower bound we got earlier gives the lemma.
\end{proof}

\begin{rem}\label{Greenrem}
This lemma exists in the literature and follows from the entropy inequality in \cite[Proposition 3.5]{Taoentropy} where an entropy version of sharpened Proposition \ref{restrictedweaksp} was obtained.  A similar sharpening can be found in (4) in \cite{pohoata2024generalized}. \cite{pohoata2024generalized} also proves abstractly that a general class of sharpenings of entropy inequalities are equivalent to their weaker forms, and comments that many known methods that prove  weaker forms also automatically prove  sharpenings.

We also comment that the $G=\Z$ case in Lemma \ref{restrictedweak} corresponds to an interesting discrete inequality.
\end{rem}

\begin{proof}[Proof of Theorem \ref{mixedBLineqforR}]
    This is a direct combination of Theorem \ref{transthm} and Lemma \ref{restrictedweak} which verifies the restricted weak type estimate.
\end{proof}

\section{
Connecting mixed-norm Brascamp-Lieb in $\Bbb{R}^2$ to Kakeya in $\Bbb{R}^d$}\label{relateto Kakeyasection}

The inequality \eqref{BLmix} has connections to the \emph{Kakeya Conjectures}. We explain this connection and make an interesting remark about a failed attempt to apply recent progress on Brascamp-Lieb to \eqref{BLmix} in this section. We first introduce the Kakeya Conjectures.

\subsection{Brief introduction to Kakeya}

On a high level, the Kakeya Conjectures assert the tubes in different directions in $\Bbb{R}^d$ cannot overlap unexpectedly much. Different forms of those conjectures then have different ways to quantify this. We introduce one  below.

\begin{defi}[Kakeya set]
    A compact set $E \subset \Bbb{R}^d$ is called a \emph{Kakeya set} if $E$ contains a unit line segment in every direction.
\end{defi}

\begin{conj}[Kakeya Conjecture, Hausdorff version]\label{KakeyaH}
     Every Kakeya set in $\Bbb{R}^d$ has Hausdorff dimension $d$.
\end{conj}

Kakeya Conjectures remain open in dimensions $d \geq 4$. They are difficult geometric measure theory problems themselves, but are surprisingly important also in modern Fourier analysis. For a good introduction of the conjectures and the above connection, see \cite{tao2001rotating}. For some recent progress on the Kakeya Conjectures, see e.g. \cite{katz2019improved, zahl2021new, hickman2022improved,  arsovski2024p} and the notable resolution of Conjecture \ref{KakeyaH} in dimension $3$ \cite{wang2022sticky, wang2025assouad, wang2025volume}.

\subsection{The connection to \eqref{BLmix}}

There is a natural connection between the mixed-norm Brascamp-Lieb inequality \eqref{BLmix} and the Kakeya Conjectures in the  $q<2$ regime. Before carefully explaining it, let us remark that the arithmetic projection technique in \cite{katz1999bounds} that we used to prove \eqref{mainthmBLineq} was a further development following Bourgain's work \cite{bourgain1999dimension} to use additive combinatorics to improve on the high-dimensional Kakeya Conjectures. This technique was developed further later in \cite{katz2002new}. It is conceivable that more results in \cite{katz1999bounds, katz2002new} can be upgraded to prove new cases of \eqref{BLmix}, just like what we did to prove \eqref{mainthmBLineq}. 

We do not try to prove more cases of \eqref{BLmix} in this paper. We chose to present  and to prove the current Theorem \ref{mixedBLineqforR} here just because it is technically easiest. Instead, we turn to a slightly different perspective: There is an abstract connection between \eqref{BLmix} for $q<2$ and Conjecture \ref{KakeyaH} (in \emph{arbitrary dimension} $d$). We now state and prove this connection.

\begin{thm}\label{BLmiximplieskakeya}
    Let $q \in (1, 2)$. Suppose there exist $\{H_j\}$ being rational subspaces of $\mathbb{R}^2$ (not equal to $x$-axis) and $p_j$ such that \eqref{BLmix} holds for this $q$ with $\frac1q+\frac1r=1$, then the Hausdorff dimension of every Kakeya set in every $\Bbb{R}^d$ is at least $\frac{d-1}{q}+1$.
\end{thm}

\begin{proof}
    For an abelian group $G$ without torsion and some $\alpha$, following \cite{katz2002new} we say we have the property $SD(\alpha)$ for $G$ if we have the following property:

   \underline{\textbf{Property.}} There is a set of rational projections $\pi_1, \ldots, \pi_k: G^2 \to G$ (by this we mean $\pi_k (g_1, g_2) = a_{1, k} g_1 + a_{2, k} g_2$ with $a_{1, k}, a_{2, k} \in \mathbb{Z}$ and $a_{1, k}, a_{2, k}$ not both zero) where $a_{1, k} \neq - a_{2, k}, \forall k$ such that the following holds: Whenever a subset $S \subset G^2$ satisfy $|\pi_j (S)| \leq N, \forall j \leq k$, we always have $|\pi_{(1, -1)} (S)| \lesssim N^{\alpha}$. Here $\pi_{(1, -1)} (g_1, g_2) = g_1 - g_2$.

    Back to the proof of the theorem, we claim that by assumption, $SD(q)$ holds for $\mathbb{Z}$. To see this, suppose some $C_j B_j (x_1, x_2) = c_j x_1 + d_j x_2, C_j, c_j, d_j \in \mathbb{Z}$ where $d_j \neq 0$. If a set $S \subset \mathbb{Z}^2$ such that all $|S_j| = |\{c_j x_1 + d_j x_2: (x_1, x_2) \in S\}| \leq N$, then by plugging in $f_j$ as $\sum_{m \in S_j} 1_{[\frac{m-1}{C_j}, \frac{m+1}{C_j}]}$ in \eqref{BLmix}, we deduce that the projection of $S$ to the first variable has size $\lesssim N^q$. But this is a linearly transformed version of $SD(q)$ for $\mathbb{Z}$ and the claim in the beginning of this paragraph is justified.

    It is standard that a finite set in $\Bbb{R}^{d-1}$ can be embedded into $\Bbb{Z}$ that preserves equalities involving a finite set of linear combinations. From this we can see that $SD(q)$ for $\Bbb{Z}$ implies $SD(q)$ for $\Bbb{R}^{d-1}$. This idea is already implicitly recorded  in the literature, see e.g. Theorem 1.1 (2) in \cite{green2019arithmetic}.

    It is already known in the literature \cite{bourgain1999dimension, katz1999bounds, katz2002new, green2019arithmetic} that $SD(q)$ for $\Bbb{R}^{d-1}$ implies the Minkowski dimension of every Kakeya set in $\Bbb{R}^d$ is at least $\frac{d-1}{q}+1$, see also the survey \cite{katz2002recent}. By recent work in \cite{green2009new, green2017new} and the very recent \cite{leng2023improved, leng2024improved} in additive combinatorics, it is possible to strengthen this argument to obtain the Hausdorff version. We are not aware of this written anywhere in the literature, so we sketch its proof in Appendix \ref{HKakeyasec}. See Theorem \ref{SDimpliesKakeyaH}.
\end{proof}

It seems reasonable to the author to make the following conjecture, which, by Theorem \ref{BLmiximplieskakeya}, will imply Conjecture \ref{KakeyaH}.

\begin{conj}\label{mixSDalpha}
    For arbitrary $\alpha > 1$, \eqref{BLmix} holds for some choices of $m, p_j, q, r$ and rational projections $B_j$ with \eqref{assumptionofqandr} satisfied and $q < \alpha$.
\end{conj}

As we see in the proof of Theorem \ref{BLmiximplieskakeya}, Conjecture \ref{mixSDalpha} will imply $SD(q)$ for $\Bbb{Z}$ for $q$ arbitrarily close to $1$ and thus imply the Kakeya Conjecture \ref{KakeyaH}.

\begin{rem}\label{katzlem}
    In \cite{katz2006elementary}, Katz proved that one cannot prove $SD(q)$ for $q < \frac{3}{2}$ by elementary graph theoretic methods (see \cite{katz2006elementary} for precise definitions). This rules out a class of methods to prove \eqref{BLmix} for $q< \frac{3}{2}$, including the Katz-Tao method  in this paper to prove \eqref{mainthmBLineq}. I am indebted to Larry Guth for pointing this out to me.
\end{rem}

\eqref{BLmix} seems less intuitive for ``genuinely'' irrational projections (i.e. ones that cannot be simultaneously well-approximated by rationals in a suitable sense).  See Section 7, Remark 4 of \cite{christ2001certain} for a more detailed discussion when $m\geq 4$ (known to be a requirement to ask an interesting question in this direction).  


\begin{question}
    Assume \eqref{assumptionofqandr}. Does \eqref{BLmix} for $q<2$ have implications on Kakeya if  $B_j$'s cannot be  simultaneously made rational under affine transformations?
\end{question}

\subsection{The absence of perturbed versions of \eqref{BLmix}}
In this subsection, we present another key difference between \eqref{BL} and \eqref{BLmix}, regarding their \emph{perturbed} versions.

First, recall that the classical Brascamp-Lieb inequality \eqref{BL} allows small perturbations of the projections and we in fact have stability of the constant \cite{bennett2018stability}. Moreover, a stronger kind of stability holds in the classical setting: If each $f_j$ is a sum of characteristic functions of unit intervals, then each $f_j \circ B_j$ is a sum of characteristic functions of (infinite) unit cylinders. It is known that whenever \eqref{BL} holds, there is $\theta > 0$ such that \eqref{BL} still holds if the direction of each cylinder is perturbed by an angle $<\theta$. Here the perturbations of directions of individual cylinders do not have to be uniform.

The above-mentioned results were typically first proved in a local version. There we take the integration over a ball $B_R$ of radius $R$ ($R \to \infty$) and allow an $R^{\varepsilon}$-loss. After this, the $R^{\varepsilon}$-loss was also removed. See \cite{bennett2006multilinear, guth2010endpoint, carbery2013endpoint, guth2015short, zhang2017endpoint, zorin2020kakeya} and a more detailed introduction in \cite{Zhang2022Brascamp}. One important case is the multilinear Kakeya inequality proved in \cite{bennett2006multilinear, guth2010endpoint}. This illustrates another connection between the Brascamp-Lieb inequality and Kakeya, not to be confused with the connection we have covered in the prior subsections.

One can wonder if a perturbed version of \eqref{BLmix} in the same way is true. This would be a much stronger assertion than \eqref{BLmix}, and would be interesting from the multilinear Kakeya perspective. The following question makes this precise:
\begin{question}\label{perturbedques}
    Do there exist some choices of $m \in \Bbb{Z}^+, p_1, \ldots, p_m, q ,r > 0$, $\theta >0$ and directions $v_1, \ldots, v_j \in S^1$ satisfying \eqref{assumptionofqandr} and $q<2$ such that
    \begin{equation}
        \|\prod_{j=1}^m (\sum_{n=1}^{N_j} 1_{T_{j, n}})^{p_j}\|_{L_x^q L_y^r (B_R)} \lesssim_{\varepsilon} R^{\varepsilon} \prod_{j=1}^m N_j^{p_j}?
    \end{equation}
    
    Here each $T_{j, n} \subset \Bbb{R}^2$ is a rectangle of infinite length and width $1$ s.t. its direction makes an angle $<\theta$ against $v_j$, and $B_R$ is any disc of radius $R>1$ in $\Bbb{R}^2$.
\end{question}





A positive answer to question \ref{perturbedques} would provide a unique perspective to study \eqref{BLmix} and Conjecture \ref{mixSDalpha}. Unfortunately, it has a \emph{negative} answer.

\begin{thm}\label{unfortunatethm}
    The answer to Question \ref{perturbedques} is negative.
\end{thm}

\begin{proof}
    We will construct counterexamples to the question where each family of tubes are parallel. The idea is inspired by \cite[Section 2]{christ2001certain}.

    Fix choices $m \in \Bbb{Z}^+, p_1, \ldots, p_m, q ,r > 0$, $\theta >0$ and directions $v_1, \ldots, v_j \in S^1$ satisfying \eqref{assumptionofqandr} and $q<2$. We want to prove it is impossible to have
    \begin{equation}\label{perturbed}
        \|\prod_{j=1}^m (\sum_{n=1}^{N_j} 1_{T_{j, n}})^{p_j}\|_{L_x^q L_y^r (B_R)} \lesssim_{\varepsilon} R^{\varepsilon} \prod_{j=1}^m N_j^{p_j}
    \end{equation}
where each $T_{j, n}$ has the property described in Question \ref{perturbedques}.

Because small perturbations are allowed, without loss of generality we may assume $v_1 \neq v_2$ and neither  $v_1$ nor $v_2$ is parallel to coordinate axes. Then by simple linear transformations, without loss of generality we can assume $v_1$ is parallel to $(1, 1)$ and $v_2$ is parallel to $(1, -1)$ (see the beginning of Section 2 of \cite{christ2001certain} for details).

Take an arbitrary irrational number $\alpha \in (1, 2)$ (for example one can take $\alpha = \sqrt{2}$ for definiteness). We can find rational numbers $r_3, \ldots, r_m$ such that the directions of $(\alpha + r_3, \alpha -r_3), \ldots, (\alpha + r_m, \alpha -r_m)$ are $\theta$-close to $v_3, \ldots, v_m$ respectively. We will see its relevance shortly.

Since $\alpha$ is irrational, there is a sequence of coprime pairs $(P, Q) \in \Bbb{Z}_+^2$ increasing in both coordinates such that
\begin{equation}\label{Dirichletqpprox}
    \left|\alpha -\frac{P}{Q}\right|< \frac{1}{Q^2}.
\end{equation}

Take a pair $(P, Q)$ satisfying \eqref{Dirichletqpprox} with $Q>10$. We now construct tubes $T_{j, n}$ to be plugged into the left hand side of \eqref{perturbed}. We will use a family of auxiliary unit balls and first construct them.

For $1 \leq i \leq P$ and $1 \leq k \leq Q$, let the point $X_{ik} = (x_{ik}, y_{ik})$ satisfy $\begin{cases}
x_{ik} - y_{ik} = iQ,\\
x_{ik} + y_{ik} = kP.
\end{cases}$ Let $\mathscr{X} = \{X_{ik}: 1 \leq i \leq P, 1 \leq k \leq Q\}$ and a set of unit balls $\mathscr{B} = \{\text{unit balls around points in } \mathscr{X}\}$.

By definition, there are $P$ lines in the direction of $(1, 1)$ whose union covers the whole $\mathscr{X}$. Choose $T_{1, 1}, \ldots, T_{1, P}$ to be the unit tubes with core lines being these. Then the union of these tubes covers every ball in $\mathscr{B}$. Similarly, we can choose unit tubes $T_{2, 1}, \ldots, T_{2, Q}$ in the direction of $(1, -1)$ such that  the union of these tubes covers every ball in $\mathscr{B}$.

For every $j \in \{3, \ldots, m\}$, we now construct tubes $T_{j, 1}, T_{j, 2}, \ldots$ parallel to $(\alpha+ r_j, \alpha - r_j)$ so that their union covers $\mathscr{B}$. To see how many tubes we need to use, we consider the possible values of $(\alpha - r_j) x_{ik} - (\alpha + r_j) y_{ik}$. This is equal to $$\alpha (x_{ik} - y_{ik}) - r_j (x_{ik} + y_{ik}) = i \alpha  Q - k r_j  P.$$

Suppose $r_j = \frac{u_j}{v_j}$. Because of \eqref{Dirichletqpprox} and $1 < \alpha < 2$, we see by the above equality that $(\alpha - r_j) x_{ik} - (\alpha + r_j) y_{ik}$ must be $O(1)$-close to some $\frac{NP}{v_j}$ where $N \in \Bbb{Z}$ and $N = O(Q v_j)$.  Hence for each $j$ we need $O(Q)$ tubes in the family $\{T_{j, n}\}$ in total so that their union covers every ball in $\mathscr{B}$. Note there are only finitely many $v_j$ and they should all be viewed as constants. We have suppressed the dependence of $v_j$ in the last assertion.

Now we are about to contradict \eqref{perturbed}. Note $P \sim Q$. For the constructed $T_{j, n}$, there is $O(Q)$ tubes in each family, and each $\sum_{n=1}^{N_j} 1_{T_{j, n}}$ is $\gtrsim 1$ on every ball in $\mathscr{B}$. Moreover, the projections of balls in $\mathscr{B}$ on the $x$-axis is a union of unit intervals around $(\frac{iQ+kP}{2}, 0)$ where $i \in [1, P]$ and $j \in [1, Q]$. Since $P$ and $Q$ are coprime, the number of these intervals is $\gtrsim Q^2$. Finally, note that every ball in $\mathscr{B}$ is in a ball of radius $R= O(Q^2)$. Thus by a similar reasoning to the proof of Lemma \ref{newnecessarylem}, we know the left hand side of \eqref{perturbed} is $\gtrsim Q^{\frac{2}{q}}$. However, the right hand side is $O(Q^{1+2\varepsilon})$. This is impossible when $Q \to \infty$ since $q<2$.
\end{proof}

We make a remark to explain some consequences of Theorem \ref{unfortunatethm}. There are three major known ways to affirmatively answer the non-mixed-norm counterpart of Question \ref{perturbedques}: the heat flow method in \cite{bennett2006multilinear}, the  method of  auxiliary polynomials \cite{guth2010endpoint} and a conceptually simpler multiscale analysis \cite{guth2015short}. However, Theorem \ref{unfortunatethm} implies they all fail if one wants to prove \eqref{BLmix} by positively answering Question \ref{perturbedques}. This is another manifestation of very different behavior of \eqref{BLmix} and its non-mixed-norm counterpart \eqref{BL}.

It may be of interest to imitate the three approaches above and see what prevents one from getting \eqref{perturbed}. If one tries to imitate the argument in \cite{bennett2006multilinear}, it seems the extrapolation has a chance to work, but there seems to be an additional term working against the estimate. For the other two methods in \cite{guth2010endpoint} and \cite{guth2015short}, the author does not know a good way to adapt them. 

\appendix

\section{$SD(\alpha)$ implies Hausdorff version of Kakeya}\label{HKakeyasec}

It remains to prove
\begin{thm}\label{SDimpliesKakeyaH}
    If $SD(\alpha)$ holds for $\Bbb{R}^{d-1}$, then the Hausdorff dimension of every Kakeya set $K \subset\Bbb{R}^d$ is at least $\frac{d-1}{q}+1$.
\end{thm}

Theorem \ref{SDimpliesKakeyaH} has a standard proof but may be of independent interest. We state it and sketch the proof in this appendix. Recall that $SD(\alpha)$ was defined in the proof of Theorem \ref{BLmiximplieskakeya}. 

To prove Theorem \ref{SDimpliesKakeyaH}, we will follow Bourgain's proof of Proposition 1.5 in \cite{bourgain1999dimension}. We remark that one useful idea in this proof is to concatenate  a few consecutive dyadic scales.

\begin{proof}[Proof sketch of Theorem \ref{SDimpliesKakeyaH}]
    The theorem can be proved in the same way as Bourgain's proof of Proposition 1.5 in \cite{bourgain1999dimension}. We only note the necessary modifications.

    We set up exactly as what Bourgain did in his argument. Instead of using Heath-Brown and Szemer\'{e}di's result right below (3.22) in \cite{bourgain1999dimension}, we can now use the result of \cite{leng2024improved} (see also \cite{green2009new, green2017new, leng2023improved} for prior results) to conclude that since the density of $Q$ is $O((\log \log M)^{-C})$, much larger than $e^{-(\log \log M)^c}$, we can find an $N_0$-terms arithmetic progression for an arbitrary $N_0$ fixed beforehand. Depending on the projections in $SD(\alpha)$ we have, we can set $N_0$ large enough so that we can use $SD(\alpha)$ in place of Bourgain's application of his Lemma 2.83 (right below (3.26) in \cite{bourgain1999dimension}) to deduce $$\mathcal{N}_{\delta} (\{a_{\xi} - b_{\xi}: \xi \in \mathcal{D}_r\}) <c \mathcal{N}_{\delta} (E_{\delta}^r)^{q+}$$ in place of Bourgain's (3.30). The conclusion then follows in the same way as \cite{bourgain1999dimension}.
\end{proof}

\bibliography{ref}{}
\bibliographystyle{alpha}
\vspace{1cm}

\noindent Ruixiang Zhang.  UC Berkeley. Email address: ruixiang@berkeley.edu

\end{document}